\documentclass[a4paper,11pt]{amsart}

\usepackage{amsfonts,amsthm,amsmath,amsfonts,latexsym,amssymb,mathtools}
\usepackage[spanish, english]{babel}
\usepackage[utf8]{inputenc}
\usepackage{graphicx,upref,calligra,float}
\usepackage[rflt]{floatflt}
\usepackage{xcolor}
\usepackage{mathrsfs}
\usepackage[T1]{fontenc}
\usepackage[colorlinks=true,linkcolor=azul,citecolor=azul]{hyperref}
\usepackage{tikz-cd}
\usepackage{lastpage}
\usepackage[bottom]{footmisc}% places footnotes at page bottom
\usepackage{type1cm} 
\usepackage{newtxtext}       % 
\usepackage[varvw]{newtxmath}   
\usepackage[backend=bibtex, style=numeric, url=false, maxbibnames=7, isbn=false]{biblatex}
\mathtoolsset{showonlyrefs}
\newtheorem{fed}{Definition}[section]
\theoremstyle{definition}
\newtheorem{teo}[fed]{Theorem}
\newtheorem*{teo*}{Theorem}
\newtheorem{lem}[fed]{Lemma}
\newtheorem{cor}[fed]{Corollary}
\newtheorem{pro}[fed]{Proposition}

\theoremstyle{definition}

\definecolor{azul}{rgb}{0.1,0.6,0.86}
\definecolor{titleblue}{rgb}{0.13,0.49,0.69}
\definecolor{mylred}{rgb}{0.85,0.24,0.2}
\definecolor{myblue}{rgb}{0,0.33,0.55}
\definecolor{myyellow}{rgb}{0.42,0.24,0.52}
\definecolor{mygreen}{rgb}{0.12,0.5,0.29}
\definecolor{myred}{rgb}{0.74,0.13,0.13}
\definecolor{mylblue}{rgb}{0.2,0.75,1}
\definecolor{mylgreen}{rgb}{0.68,0.98,0.6}
\definecolor{mylyellow}{rgb}{0.86,0.85,0.55}
\definecolor{myllyellow}{rgb}{0.87,0.86,0.56}
\definecolor{naranja}{RGB}{249,153,96}
\definecolor{sidebardarkcolor}{rgb}{0.21,0.31,0.40}
\definecolor{sidebarlightcolor}{rgb}{0.7,0.77,0.836}

\def\bdem{\begin{proof}}
	\def\edem{\end{proof}}

\def\fii{\varphi }
\def\la{\lambda}
\def\La{\Lambda}

\def\w{\omega}
\def\W{\Omega}

\def\N{\mathbb{N}}
\def\Z{\mathbb{Z}}
\def\R{\mathbb{R}}

\def\C{\mathbb{C}}

\def\T{\mathbb{T}}

\def\ele{\mathcal{L}}

\def\ete{\mathcal{T}}

\newcommand{\peso}[1]{ \quad \mbox{  #1 } \quad }

 \DeclareMathOperator{\tr}{tr}

\newcommand{\esssup}{{\mathrm{ess}\sup}}

\newcommand{\supp}{{\rm supp\,}}

\begin{document}
	
	\title{Variations on two Cabrelli's works}
	
	\author{Elona Agora}
	\address{ CUNEF Universidad, Madrid, Spain}
	\email{elona.agora@cunef.edu}
	
	\author{Jorge Antezana}
	\address{ Barcelona University, Barcelona, Spain}
	\email{jorge.antezana@ub.edu}
	
	\author{Diana Carbajal}
	\address{ University of Vienna, Vienna, Austria }
	\email{diana.agustina.carbajal@univie.ac.at}

	\date{}
	
%	\institute{Elona Agora \at CUNEF Universidad, Madrid, Spain, \email{elona.agora@cunef.edu} \and Jorge Antezana \at Barcelona University, Barcelona Spain. \email{jorge.antezana@ub.edu} \and Diana Carbajal   \at University of Vienna, Vienna, Austria \email{diana.agustina.carbajal@univie.ac.at}}
	
	\maketitle
	
	\begin{center}
		\textit{To our dear friend, Carlos Cabrelli}
	\end{center}
	
	\medskip
	
	\begin{abstract}
		In this paper we present two different problems within the framework of shift-invariant theory. First, we develop a triangular form for shift-preserving operators acting on finitely generated shift-invariant spaces. In case of the normal operators, we recover a diagonal decomposition. The results show, in particular, that any finitely generated shift-invariant space can be decomposed into an orthogonal sum of principal shift-invariant spaces, with additional invariance properties under a shift-preserving operator.
		Second, we provide a new characterization of the multi-tiling sets $\Omega\subset\R^d$ of positive measure for which $L^2(\Omega)$ admits a structured Riesz basis of exponentials that is formulated in the ambient space $\T^{k\times k}$. In addition, we show a simpler sufficient condition which generalizes the \textit{admissibility} property, that is also necessary for 2-tiling sets.
	\end{abstract}

	\section{Introduction}
	
	In this contribution, we address two different yet connected problems that are related to the theory shift-invariant spaces, and are  inspired by some developments from the prolific work of Cabrelli \cite{ACCP,AAC,CC,CUC}.
	
	Throughout this paper, a shift-invariant space is a subspace $V$  of $L^2(\R^d)$ that is invariant under translations of a certain lattice $H\subset \R^d$; that is, $T_hV\subset V$ for every $h\in H$, where $T_hf = f(\cdot - h)$. These spaces play a central role in areas within harmonic analysis such as wavelet and Gabor theory, approximation theory, sampling theory and signal processing.  Many questions formulated in shift-invariant spaces become tractable through tools such as the \textit{fiberization mapping} and the \textit{range function} (c.f. Def. \ref{isometria} and \ref{range-function}, respectively) \cite{He64,BDR2,RS,B}.
	When the space is generated by the shifts of a finite family of functions, its global complexity is essentially encoded by a measurable field of finite-dimensional spaces (the \textit{fiber spaces}) indexed by a compact domain. This perspective reduces infinite-dimensional problems to parametrized families of finite-dimensional ones, making it possible to tackle the questions using linear algebraic techniques. 
	
	\medskip
	
	\noindent In this paper, we focus on the following two problems:
	\begin{enumerate}
		\item[(P1)]\label{p1} \textit{Develop a triangular form for shift-preserving operators acting on finitely generated shift-invariant spaces.}
		\item[(P2)]\label{p2} \textit{Provide a new characterization of the multi-tiling sets $\Omega\subset\R^d$ of positive measure for which $L^2(\Omega)$ admits a Riesz basis of exponentials with a periodic set of frequencies.}
	\end{enumerate}
	
	\medskip
	
	\noindent The first problem (P1) builds upon the theory of shift-preserving operators. These are the operators acting on shift-invariant spaces which commute with the shifts of the underlying lattice. The analysis of such operators can be reduced to studying their fiber representations via the \textit{range operator} (c.f. Def. \ref{range operator}), a field of linear transformations that act on the fiber spaces of the underlying shift-invariant space. Many global properties of a shift-preserving operator are reflected by pointwise properties of the range operator, holding uniformly almost everywhere \cite{B,BI2019}. 
	
	One advantage of this fiber perspective is the possibility to derive global canonical decompositions of those operators that act on finitely generated shift-invariant spaces, through the analysis of the local decompositions of the range operator. An initial approach to this problem was presented in \cite{ACCP}, where a diagonal-type decomposition for normal shift-preserving operators was developed. More precisely, let $V\subset L^2(\R^d)$ be a finitely generated shift-invariant space under translations of a lattice $H\subset \R^d$, and let $L:V\to V$ be a normal shift-preserving operator. Then $L$ can be decomposed into an orthogonal sum of operators densely defined as
	$$\La_a = \sum_{h\in H} a(h) T_h,$$
	where $a={a(h)}_{h\in H}\in \ell^2(H)$ and its Fourier transform is essentially bounded. These operators are often called \textit{$s$-eigenvalues} of $L$ in \cite{ACCP} (c.f. Def. \ref{s-eigenval}). 
	
	In (P1) we introduce a new type of canonical decomposition for (not necessarily normal) shift-preserving operators acting on finitely generated shift-invariant spaces, which we call \textit{triangularization}. Our main result (Theorem \ref{triangular}) shows that there exist shift-invariant subspaces of $V$, 
	$$
	V=V_\ell\supsetneq \dots\supsetneq V_{1},
	$$ 
	such that for every $1\leq j\leq\ell$,  $\ele(V_j)=j$, and $ L(V_j)\subseteq V_j$. Here, $\ele(W)$ denotes the smallest number of generators of a shift-invariant space $W$. The proof mirrors the ideas behind the classical Schur decomposition of finite-dimensional linear algebra. In the particular case where the operator is normal, we use the triangular form to deduce a type of diagonal decomposition (Corollary~\ref{spectral-cor}). Specifically, we show that there exists an orthogonal decomposition of $V$ into \textit{principal} shift-invariant subspaces which are \textit{reducing} for $L$. As a consequence, we are able to recover the main result of \cite{ACCP}. Our canonical decomposition results extend, for finitely generated spaces, the classical decomposition theorem of \cite{B} by incorporating the additional structure induced by the invariance under a certain shift-preserving operator.
	
	\medskip
	
	\noindent The second problem (P2) deals with the central question of the existence of bases of exponential functions. The main interest in studying this question comes from the classical sampling theory in Paley-Wiener spaces. Given a measurable set~$\Omega\subset\R^d$, the space $PW_\Omega$ denotes the subspace of $L^2(\R^d)$ consisting of the functions whose Fourier transform is supported on $\Omega$, often called the \textit{spectrum}. These spaces are a particular example of shift-invariant spaces, as they are invariant under any translation. A discrete set $\Gamma\subset\R^d$ is a sampling and interpolation set for $PW_\Omega$ if and only if the system $$E(\Gamma):=\left\{e_\gamma\,:\,\gamma\in\Gamma\right\} \quad \text{where}\quad e_\gamma(\w) := e^{2\pi i \gamma \omega} $$ is a  {\it Riesz basis of exponentials} of $L^2(\Omega).$
	This means that $E(\Gamma)$ is complete in $L^2(\Omega)$ and satisfies that
	\begin{equation}\label{eq:RBE}
		A\,\sum_{\gamma\in\Gamma} |c_\gamma|^2\,\leq\;\;\; \left\|\sum_{\gamma\in\Gamma} c_\gamma e^{2\pi i \gamma\w}\right\|^{2}_{L^2(\Omega)}\leq \,B\,\sum_{\gamma\in\Gamma}|c_\gamma|^2,\quad\forall \{c_\gamma\}\in\ell^2(\Gamma),
	\end{equation}
	for some positive constants $A,B$. From the sampling point of view, the set $\Gamma$ has the property that, for any function $f\in PW_\Omega$, if we know the values of the function $f$ at the points of $\Gamma$, then we have enough information to reconstruct the function uniquely. 
	
	While orthonormal systems of exponentials are the most desirable type of basis for $L^2(\Omega)$, their existence imposes a very rigid structure on $\Gamma$ and $\Omega$. In fact, it has been proved that there exist sets $\Omega\subset\R^d$ that do not admit such a basis~\cite{Fu01,Ko00,La01}. The study of these systems is closely linked to Fuglede's conjecture \cite{Fu74}, which claimed that a set $\Omega$ admits an orthogonal basis of exponentials if and only if it tiles the space by translations along some discrete set $\Lambda\subset \R^d$, that is, $$\sum_{\lambda\in\Lambda} \chi_{\Omega}(\w+\lambda) = 1 \quad \text{for a.e. }\w\in\R^d.$$ 
	However, the conjecture was disproved in both directions for dimensions $d\geq 3$ \cite{T04, KM06, FMM06, FR06} and it remains open for $d=1,2$. Despite these counterexamples, the conjecture does hold in several important cases \cite{La01,IKT03, LM22}. Notably, Fuglede himself proved that if $\Lambda$ is a full lattice in $\R^d$, then $\Omega$ tiles $\R^d$ by translations on $\Lambda$ if and only if the system $E(H)$ forms an orthogonal basis of $L^2(\Omega)$, where $H$ is the dual lattice of $\Lambda$ \cite{Fu74}. 
	
	Given these limitations, it is natural to turn to more flexible systems. Riesz bases of exponentials provide a robust alternative to orthonormal sets: although they lack orthogonality, they still allow stable and unique representations in $L^2(\Omega)$. However, determining whether a given set admits such a basis remains highly nontrivial. To date, only a few positive examples are known \cite{Ma06,GL14, Kol, KN15, KN16, GL18, DL22}, and remarkably, only one example has been found not to admit a Riesz basis of exponentials \cite{KNO23}.

	In this paper we focus on a particular family of sets known to admit Riesz bases of exponentials: the \textit{multi-tiling} sets. A set $\Omega\subset \R^d$ multi-tiles $\R^d$ at level $k$ by translations on a discrete set $\Lambda$ (or $k$-tiles $\R^d$) if $$\sum_{\lambda\in\Lambda}\chi_{\Omega}(\w+\lambda)=k\quad\text{for a.e. }\w\in\R^d.$$
	It was proved in \cite{GL18} that every \textit{Riemann integrable} set $\Omega$ which $k$-tiles $\R^d$ by translations on a full lattice $\Lambda$ admits a Riesz basis of exponentials. In  \cite{Kol},  Kolountzakis established the same conclusion for every \textit{bounded} set $\Omega$ with this tiling property. Moreover, he showed that the basis can be chosen with a periodic set of frequencies
	$$
	\{e_{a_j+h}\,:\,h\in H, \,j=1,\dots,k\},
	$$
	for some $a_1,\dots,a_k\in\R^d$, where $H$ is the dual lattice of $\Lambda$. We refer to it as  a {\it structured} system of exponentials. This result was later extended to the context of locally compact abelian (LCA) groups in \cite{AAC}. The extension was based on the theory of shift-invariant spaces and fiberization techniques. This perspective also allowed the authors to establish in \cite{AAC}  that the converse is true; that is, the existence of a structured Riesz basis of exponentials for $L^2(\Omega)$ implies that $\Omega$ is a multi-tiling set. 
	
	The question of whether this result remains valid for {\it unbounded} multi-tile sets of finite measure was answered  in the negative also in \cite{AAC}, where the authors constructed an explicit example of an unbounded $2$-tile set of $\R$ by translations on $\Z$ that does not admit a structured Riesz basis of exponentials. However, the authors in \cite{CC} observed that a special class of (not necessarily bounded) multi-tiling sets, namely, those which are \textit{admissible} \eqref{admissible}, always support a structured Riesz basis. The class of admissible multi-tiles contains, in particular, all the bounded multi-tiles. 
	
	Admissibility is, nevertheless, only a sufficient condition, as remarked in \cite{CC}. This raised the problem of precisely characterizing the class of multi-tiles that support a structured Riesz basis of exponentials. This question was fully settled in \cite{CUC} through a characterization involving the Bohr compactification of the underlying lattice. While this result is both general and elegant, the condition requires verifying non-trivial properties in a compactified space of highly abstract nature. 
	
	In this paper, we present an alternative characterization of those multi-tiling sets that admit a structured Riesz basis of exponentials. While our findings are inspired by the techniques in \cite{CUC}, we are able to formulate an equivalent condition in the more tractable ambient space $\T^{k\times k}$ (see Theorem \ref{nueva caracterizacion}). The condition requires verifying that the determinant function is not null on a closed subspace of $\T^{k\times k}$ which depends on $\Omega$ and the vectors $a_1,\dots,a_k\in\R^d$. Moreover, we introduce a new simple geometric property (see Def. \ref{def:sep}) that relaxes the admissibility condition, and show that the multi-tiling sets for which this property holds admit a structured Riesz basis of exponentials (Theorem \ref{separados}). Notably, this condition is also necessary for the special case of $2$-tiling sets (Theorem \ref{necesario para 2tilings}).
	
	\medskip
	
	\noindent We remark that all the results in this paper can be extended to the LCA group context with minimal  modifications.
	
	\medskip
	
	\noindent The paper is organized as follows. In Section \ref{sec-2} we review the main ideas of the theory of shift-invariant spaces that will be used throughout the paper. In Section \ref{sec-3}, we focus on the first postulated problem (P1). Finally, we devote Section \ref{sec-4} to problem (P2).

	\section{Shift-invariant spaces}\label{sec-2}
	
	In this section, we recall some of the basic facts about shift-Invariant spaces (SIS), which will be used throughout this article. The interested reader is referred to \cite{B, BDR2,CP, Chr2003} for the proofs, as well as much more information about these spaces. 
	
	Let $\La\subset\R^d$ be a  \textit{full lattice}, i.e. there exists $M\in GL_d(\R)$ such that $\La = M\Z^d$. A  \textit{fundamental domain} with respect to a lattice $\La$ is a Borel set of representatives of the quotient group $\R^d/\La$. In our setting, we will always take the set $I=M\left([-\frac12,\frac12)^d\right)$ as a fundamental domain. The  \textit{dual lattice} of $\La$ is the set
	$$H:=\{h\in\R^d\,:\, \langle h,\la \rangle \in \Z, \; \text{for all} \, \la\in\La\},$$
	which coincides with $H=(M^{t})^{-1}\Z^d$.
	Throughout the remainder of this work, we will fix a lattice $\La$ and $H$ will always denote its dual.
	
	\begin{fed}
		We say that a closed subspace $V\subset L^2(\mathbb R^d)$ is shift-invariant under translations of $H$ or {\it $H$-invariant} if for each $f\in V$ we have that $T_h f\in V$, $\forall\,h \in H$, where $T_h f(x)= f(x-h)$ is the translation by $h$.
	\end{fed}
	
	Given a countable set of functions $\Phi\subset L^2(\mathbb R^d)$, we denote by 
	$$S(\Phi) := \overline{\text{span}} \left\{ T_h\varphi\,:\,\varphi \in \Phi, \,h\in H \right\}$$
	the $H$-invariant space generated by $\Phi$. 
	When the translations of the set $\Phi$ by $H$ form a Parseval frame of $S(\Phi)$ we simply say that $\Phi$ is a Parseval frame generator of $S(\Phi)$.
	When $\Phi$ is a finite set, we say that $V=S(\Phi)$ is a  \textit{finitely generated} shift-invariant space and, when $\Phi=\{\varphi\}$ for some function $\varphi$, we say that $V$ is a  \textit{principal} shift-invariant space. The  \textit{length} of a finitely generated shift-invariant space $V\subset L^2(\R^d)$ is denoted by $\mathcal{L}(V)$ and is defined as the smallest natural number $\ell$ such that there exists $\varphi_{1},...,\varphi_{\ell} \in V$ with $V=S(\varphi_{1},...,\varphi_{\ell})$.

	Much of the structure of shift-invariant subspaces of $L^2(\R^d)$ can be understood through the  \textit{fiberization mapping}. This key transformation associates each function $f\in L^2(\R^d)$ with an $\ell^2(\La)$-valued square-integrable function on $I$. Explicitly, we consider the Hilbert space $$L^2\left(I,\ell^2(\La)\right),$$ which consists of all vector-valued measurable functions $\psi: I\rightarrow \ell^2(\La)$ such that
	$$\|\psi\|=\left( \int_{I} \|\psi(\omega)\|^2_{\ell^2} \,d\omega\right)^{1/2} < \infty.$$
	The mapping is defined as follows.
	
	\begin{fed}\label{isometria}
		The fiberization mapping $\mathcal T: L^2(\mathbb R^d) \to L^2\left(I,\ell^2(\La)\right)$ is defined by
		$$
		\mathcal T f(\omega) = \{\hat{f} (\omega+\la) \}_{\la \in  \La},
		$$
		where $\hat{f}$ denotes the Fourier transform of $f$.
		The sequence $\mathcal T f(\omega)$ is referred to as the fiber of $f$ at $\omega \in I$. 
	\end{fed}
	
	Observe that the fiberization mapping is an isometric isomorphism, which is related to the periodization identity  
	$$
	\int_{\R^d} f(x) \,dx =  \int_{I} \sum_{\la\in\La} f(\w+\la)\,d\w, \quad  f\in L^1(\R^d).
	$$
	
	This fiber-wise perspective naturally leads to the concept of  \textit{range function}.
	
	\begin{fed}\label{range-function}
		A {\it range function} is a mapping
		\begin{align*}
			J:I&\rightarrow\{\text{\,closed subspaces of }\ell^2(\La)\,\}\\
			\omega&\mapsto J(\omega).
		\end{align*}
		By identifying the subspaces of $\ell^2(\La)$ with the corresponding orthogonal projector, we say $J$ is measurable if the scalar function $\omega\mapsto\langle P_{J(\omega)}u,v \rangle$ is measurable for every $u,v\in\ell^2(\La)$. Here, $P_{J(\omega)}$ is the orthogonal projection of $\ell^2(\La)$ onto $J(\omega)$. 
	\end{fed}

	The significance of range functions is highlighted by a fundamental characterization theorem, originally due to Helson \cite[Theorem 8]{He64}. It states that every shift-invariant space can be described as the collection of functions whose fibers lie in a measurable family of subspaces, precisely those defined by a range function. We present here the characterization as stated in \cite[Proposition 1.5]{B}.
	\begin{teo}\label{thm:characterization-sis}
		A closed subspace $V\subset L^{2}(\mathbb{R}^{d})$ is shift-invariant if and only if there exists a measurable range function such that
		\begin{equation}\label{eq: associated range function}
			V=\{f\in L^{2}(\mathbb{R}^{d}): \mathcal{T}f(\omega) \in J(\omega) \text{ for a.e. } \omega \in I\}.
		\end{equation}
		Furthermore, if $V=S(\Phi)$ for some coun\-table set $\Phi\subset L^2(\mathbb R^d)$, then 
		\begin{equation*}
			J(\omega) = \overline{\text{span}}\{\mathcal T f(\omega)\,:\,f\in\Phi\,\}
		\end{equation*} 
		for a.e. $\omega\in I$.
		Under the convention that two range functions are identified if they are equal a.e. $\omega\in I$, the correspondence between $V$ and $J$ is one-to-one. We call the subspace $J(\omega)$ the  \textit{fiber space} of $V$ at $\omega$. 
	\end{teo}
	
	Rather than analyzing $V$ directly in  $L^2(\R^d)$, one may study its properties through its range function in $\ell^2(\La)$. The following lemma summarizes some classical results that illustrate the strength of this perspective. For detailed proofs, see  \cite{BDR2} and \cite{AC}.  
	
	\begin{lem}\label{range-properties}
		Let $V, U$ be shift-invariant spaces of $L^2(\R^d)$ with associated range functions $J_V, J_U$, respectively. Then:
		\begin{enumerate}
			\item[i)]\label{item 1 range-properties} The orthogonal complement  $V^{\perp}$ is shift-invariant and $J_{V^{\perp}}(\w)=\left(J_V(\w)\right)^{\perp}$ for a.e. $\w\in I $.
			\item[ii)]\label{item 2 range-properties} The space  $V\cap U$ is  shift-invariant and its range function satisfies $J_{V\cap\, U}(\omega) = J_V(\omega) \cap J_U(\omega),$
			for a.e. $\omega\in I$.
			\item[iii)]\label{item 3 range-properties} If $V$ and $U$ are orthogonal, then $V\oplus U$ is shift-invariant and its range function is given by $J_{V\oplus U}(\omega) = J_V(\omega) \oplus J_U(\omega),$ for a.e. $\omega\in I$.
			\item[iv)] The length of $V$ can be characterized as follows:
			$$\mathcal{L}(V) = \underset{\omega\in I}{\text{\rm ess sup }}  \dim J(\omega).$$
		\end{enumerate}
	\end{lem}
	
	The last item points out a key property of shift-invariant spaces. If the space $V$ is finitely generated, then the associated fiber spaces $J(\w)$ are finite-dimensional for a.e. $\w\in I$. 
	
	Let us denote the {\it spectrum} of $V$ as $$ \sigma(V) := \left\{\omega \in I\,:\,  J(\omega) \neq\{0\} \right\}.$$
	The next  lemma follows from \cite[Proposition 2.9 and Lemma 3.6]{BDR2}.  
	
	\begin{lem}\label{spec principal}
		Let $V$ be a shift-invariant space of $L^2(\R^d)$. Then, there exists $\varphi\in V$ such that $\supp \|\mathcal T \varphi(\cdot)\|_{\ell^2}=\sigma(V).$
	\end{lem}
	
	Finally, since the dimension of $J(\w)$ may vary depending on $\w\in I$, it will be  convenient to split $I$ into measurable sets where the dimension of $J(\w)$ is constant. The lemma below is an immediate consequence of the measurability of the function $\w\mapsto\tr(J(\w))$, where $\tr$ denotes the trace mapping.
	
	\begin{lem}\label{lem:dimensions}
		Let $V$ be a shift-invariant space of length $\ele(V) = \ell$ and with corresponding range function $J$. Then, the sets
		$$
		A_n := \{\w\in  I\,:\, \dim J(\w) = n\}, \quad n=0, \dots, \ell,
		$$
		are measurable and satisfy 
		$$
		I = \bigcup_{n=0}^{\ell} A_n.
		$$
	\end{lem}

	\section{Canonical forms for shift-preserving operators}\label{sec-3}
	
	The fiberization mapping not only provides insight into the structure of shift-invariant spaces but also serves as a powerful tool to study the operators acting on them. In particular, shift-preserving operators, those that commute with shifts along a lattice, play a fundamental role in this context. Through the lens of fiberization, such operators can be understood in terms of range operators, acting pointwise on the fibers. When the shift-invariant space is finitely generated, this perspective enables the use of finite-dimensional linear algebra techniques to achieve a deeper spectral analysis and 
	to construct canonical forms such as triangular and diagonal representations.
	
	In this section, we focus on the first problem mentioned in the introduction (P1). We prove that any shift-preserving operator acting on a finitely generated shift-invariant space admits a triangular decomposition. When the operator is normal, we deduce a diagonal decomposition. These results ultimately lead to an alternative proof of the spectral decomposition theorem from \cite{ACCP} and extend the decomposition theorem \cite[Theorem 3.3]{B}.

	\subsection{Shift-preserving operators}\label{section-SP}
	
	In this section, we recall some definitions and properties of shift-preserving operators.
	
	\begin{fed}
		Let $V\subset L^2(\mathbb R^d)$ be an $H$-invariant space and $L:V\rightarrow L^2(\mathbb R^d)$ be a bounded operator. We say that $L$ is {\it shift-preserving} or $H$-preserving if it commutes with translations, i.e. $$LT_h = T_h L,\quad\text{ for all }h\in H.$$
	\end{fed}
	
	In \cite{B}, Bownik analyzed the properties of these operators by  introducing the concept of a  \textit{range operator}, which provides a way to understand the behavior of a shift-preserving operator via the fiberization mapping.
	
	\begin{fed}\label{range operator}
		Given two measurable range functions $J,J'$, a range operator $R: J\to J'$ is a choice of linear operators $R(\w):J(\w) \to J'(\w)$, $\w\in  I$.
		We say that $R$ is {\it measurable} if  $\omega \mapsto \langle R(\omega) P_{J(\omega)} u,v\rangle$ is a measurable scalar function for every $u,v\in\ell^2(\La)$.
	\end{fed}
	
	This notion leads to the following fundamental correspondence.
	
	\begin{teo}[{\cite[Theorem 4.5]{B}}] Let $V, V' \subseteq L^2(\R^d)$ be two shift-invariant subspaces with range functions $J$ and $J'$, respectively.
		Given a bounded, shift-preserving operator $L:V\rightarrow V'$, there exists a measurable range operator $R: J \to J'$ such that 
		\begin{equation}\label{prop of R}
			(\mathcal T\circ L) f(\omega) = R(\omega) \left(\mathcal T f(\omega)\right),
		\end{equation}
		for a.e. $\omega\in I$ and $f\in V$.
		
		Conversely, given a measurable range operator $R: J \to J'$ with $$\underset{\omega\in  I}{\esssup} \|R(\omega)\|<\infty,$$ there exists a bounded shift-preserving operator $L: V\to V'$ such that  \eqref{prop of R} holds. The correspondence between $L$ and $R$ is one-to-one under the convention that the range operators are identified if they are equal a.e. $\w\in I$. 
	\end{teo}
	
	As anticipated, the global properties of shift-preserving operators are reflected by the pointwise properties of the corresponding range operators. The following theorem summarizes this relationship. See \cite{B} and \cite{BI2019} (from the multiplication-invariant operators' perspective) for the detailed proofs. 
	
	\begin{teo}\label{thm:pointwise}
		Let $V,V'\subset L^2(\R^d)$ be two shift-invariant spaces with range functions $J_V$ and $J_{V'}$, respectively. Let $L:V\rightarrow V'$ be a shift-preserving operator with corresponding range operator $R: J_V\rightarrow J_{V'}$. Then the following are true:
		\begin{enumerate}
			\item[i)]\label{op-norm} $\|L\|_{\text{op}}=\esssup_{\w\in I} \|R(\w)\|_{\text{op}}$. 
			\item[ii)] The adjoint $L^*:V'\rightarrow V$ is also shift-preserving with corresponding range operator $R^*:J_{V'}\rightarrow J_V$ given by $R^*(\w)=(R(\w))^*$ for a.e. $\w\in I$.
			\item[iii)] $L$ is normal (self-adjoint) if and only if $R(\w)$ is normal (self-adjoint) for a.e. $\w\in I$.
			\item[iv)] $L$ is injective if and only if $R(\w)$ is injective for a.e. $\w\in  I$.
			\item[v)]\label{partial-isom} $L$ is a (partial) isometry if and only if $R(\w)$ is a (partial) isometry for a.e. $\w\in I$.
			\item[vi)]\label{rank} The space $V''=\overline{L(V)}\subseteq L^2(\R^d)$ is shift-invariant and its range function is given by $$J_{V''}(\w)=\overline{R(\w) J_V(\w)},$$
			for a.e. $\w\in I$.
			\item[vii)]\label{ker} The space $\ker(L)$ is shift-invariant and its range function is given by $K(\w)=\ker(R(\w))$ for a.e. $\w\in I$.
		\end{enumerate}
	\end{teo}

	\subsection{Triangular form for shift-preserving operators}
	
	We henceforward focus on the shift-preserving operators that act on shift-invariant spaces of finite length, i.e., $L:V\to V$ with $\ele(V) =\ell < \infty$. Then, for almost every $\w$, the fiber space $J(\w)$ is a finite-dimensional subspace of $\ell^2(\La)$, and the operator $R(\w)$ acts on a finite-dimensional space.
	By working within the fiber domain, we establish a canonical triangular form that mirrors the classical Schur decomposition. This decomposition provides a foundational step toward understanding the spectral structure of such operators.
	
	The central idea is that measurable selections of eigenvalues of the range operator yield shift-invariant subspaces that are invariant under the operator.
	For this purpose, we first state and prove the following proposition.
	
	\begin{pro}\label{An eigenvalue for R}
		Let $V$ be a finitely generated shift-invariant space and let $L:V\rightarrow V$ be a shift-preserving operator with associated range operator $R$. Then, there exists a measurable function $\la: I\to \C$ such that $\la(\w)$ is an eigenvalue of $R(\w)$ for a.e. $\w\in I$. Moreover, there exists a function $\psi_\la \in V$ such that $\supp \| \mathcal T \psi_\la(\w)\| = \sigma(V)$ and
		\begin{equation}\label{eq:eigenfunction}
			R(\w)\mathcal T \psi_\la(\w)=\la(\w) \mathcal T \psi_\la(\w), \quad \text{for a.e.} \, \w \in I.
		\end{equation}	
	\end{pro}
	
	\begin{proof}
		Let $\ell = \mathcal L (V)$ and let $A_n = \{ \w\in  I\,:\, \dim J(\w) = n\}$, with $n=0,\dots,\ell$. By Lemma \ref{lem:dimensions}, these sets are measurable and partition the set $ I$.
		Fix $1\leq n\leq \ell$. The existence of a measurable function $\la_n: A_n\to \C$ such that $\la_n(\w)$ is an eigenvalue of $R(\w)$ for a.e. $\w\in A_n$  is due to Azoff \cite[Corollary 4]{A} (see also \cite[Theorem 5.4]{ACCP}). We glue them all to obtain the desired function $\la: I\to \C$ as it follows:
		$$
		\la(\w):=\left\{ \begin{array}{lcc} \la_n (\w)& \quad \text{ if }  & \w  \in A_n, \\ \\ 0  &  \quad \text{ if } & \w \in A_0. \end{array} \right.
		$$
		
		Define the measurable range function $J_\lambda(\w):= \ker(R(\w)- \lambda(\w)\mathcal{I}_\w) \subset J(\w)$, where $\mathcal I_\w$ denotes the identity operator on $J(\w)$, and consider $V_\la\subset V$ as the unique shift-invariant space associated to $J_\lambda$ in the sense of Theorem \ref{thm:characterization-sis}. By construction, $\sigma(V) = \sigma(V_\la)$.
		It follows from Lemma \ref{spec principal} that there exists a function $\psi_\la\in V_\la$ such that  $\supp \| \mathcal T \psi_\la(\w)\| = \sigma(V)$ and satisfies \eqref{eq:eigenfunction}.
	\end{proof}

	We are ready to state and prove the triangular decomposition.
	
	\begin{teo}\label{triangular}
		Let $L :V\rightarrow V$ be a bounded shift-preserving operator on a finitely generated shift-invariant space $V$ of length $\ell$.
		Then there exist shift-invariant subspaces of $V$, 
		$$
		V=V_\ell\supsetneq \dots\supsetneq V_{1},
		$$ 
		such that for every $1\leq j\leq\ell$,  $\ele(V_j)=j$, and each $V_j$ is \textit{invariant} for the operator $L$, that is, 
		$ L(V_j)\subseteq V_j$. Moreover, $V$ admits an orthogonal decomposition
		$$
		V = S(\psi_1)\oplus \dots \oplus S(\psi_\ell),
		$$
		where each $\psi_j$ is a Parseval frame generator of $S(\psi_j)$,  with the spectra satisfying the inclusions 
		$
		\sigma(S(\psi_{j+1}))\subset \sigma(S(\psi_j)),
		$
		and
		$$
		V_j = \bigoplus_{i=1}^{j} S(\psi_i), \quad 1\leq j \leq \ell.
		$$
	\end{teo}
	
	\begin{proof} We argue by induction on $\ell$.  If $\ell=1$, the result holds since $V$ is already a principal subspace.  
	Assume now that $\ell=n>1$, and that the result holds for all spaces of length strictly smaller than $n$. 
	Let $R$ be the range operator corresponding to $L$. By Proposition \ref{An eigenvalue for R}, 
	there exists a measurable function $\la_1 :I\to\C$ and a nonzero function $\psi_1\in V$ such that
	$$
	R(\omega)\mathcal T \psi_1(\omega) = \la_1(\omega)\,\mathcal T \psi_1(\omega), 
	\qquad \text{for a.e. }\omega \in I,
	$$
	and 
	$$
	\supp \|\mathcal T \psi_1(\cdot)\| = \sigma(V).
	$$
	In particular $\sigma(S(\psi_1))=\sigma(V)$. Thus $\psi_1$ defines the first generator of our decomposition. Moreover, without loss of generality, we may assume that
	$$
	\|\mathcal T\psi_1(\omega)\|=1, \qquad \text{for a.e. }\omega\in\sigma(S(\psi_1)),
	$$
	so that $\psi_1$ is a Parseval frame generator of $S(\psi_1)$.  
	The principal shift-invariant space
	$
	S(\psi_1)\subset V
	$
	is also invariant under $L$. Let $W:= V\ominus S(\psi_1)$. By Lemma \ref{range-properties}, $W$ is shift-invariant and the operator $L_W:W\to W$ defined by
	$$
	L_W := P_{W} L P_{W}
	$$
	is shift-preserving. Moreover, since $\text{span}\, \{ \mathcal{T} \psi_1 (\w) \}\neq \{0\}$ for a.e. $\w\in \sigma(V)$, and $\ele(V) = \ell$, we have that $\ele(W )= \ell-1$. 
	By the inductive hypothesis applied to $L_W:W\to W$, there exist orthonormal generators $\psi_2,\dots,\psi_\ell$ in $W$ such that
	$$
	W = S(\psi_2)\oplus\dots\oplus S(\psi_\ell),
	$$ 
	and the inclusions 
	$$
	\sigma(S(\psi_{i+1})) \subset \sigma(S(\psi_i)), \qquad 1\le i\le \ell-1,
	$$
	hold. Moreover, the spaces 
	$$
	W_j:= \bigoplus_{i=2}^j S(\psi_i), \quad  j\in \{2,\ldots,\ell\},
	$$
	are invariant under $L_W$ and, therefore invariant under $L$. Now define $V_1= S(\psi_1)$ and
	$$
	V_j= S(\psi_1) \oplus W_j=\bigoplus_{i=1}^{j} S(\psi_i), \quad  j\in \{2,\ldots,\ell\}.
	$$
	By construction, it follows that
	$$
	V = V_\ell \supsetneq V_{\ell-1} \supsetneq \cdots \supsetneq V_{1},
	$$
	where each $V_j$ is shift-invariant, invariant under $L$, and satisfies $\mathcal L(V_j)=j$. 
	Since  $\sigma(S(\psi_1))=\sigma(V)$,  $\sigma(S(\psi_2))\subseteq\sigma(S(\psi_1))$, we also have that
	$$
	\sigma(S(\psi_{i+1})) \subset \sigma(S(\psi_i)), \qquad 1\le i\le \ell-1.
	$$ 
\end{proof}

We now show that when the acting operator $L:V\to V$ is normal, the shift-invariant space can be decomposed into an orthogonal sum of principal shift-invariant spaces, each of them reducing for $L$. First, we prove the next lemma. Recall that a subspace $U$ is  \textit{reducing} for an operator $L$ if $L(U)\subseteq U$ and $L(U^\perp) \subseteq U^\perp$.

\begin{lem}
	Let $L:V\to V$ be a normal bounded shift-preserving operator acting on the shift-invariant space $V$ with finite length $\ele(V)=\ell$. Let $U$ be a shift-invariant subspace of $V$ which is invariant for $L$, then $U$ is reducing for $L$.
\end{lem}

\begin{proof}
	In the finite-dimensional setting, every invariant subspace of a normal operator is a reducing subspace of such operator. Let $J$ be the range function associated to $V$ and $R:J\to J$ the corresponding range operator of $L$. By Lemma \ref{range-properties}, $R(\w):J(\w)\to J(\w)$ is a normal operator acting on a finite-dimensional space for a.e. $\w\in I$. Let $J_U$ be the range operator associated to $U$, then $J_U(\w)$ and $(J_U(\w))^\perp$ are invariant for $R(\w)$ for a.e. $\w\in  I$. But $(J_U(\w))^\perp = J_{U^\perp}(\w)$ and hence we deduce that $U^\perp$ is invariant for $L$.
\end{proof}

\begin{cor}\label{spectral-cor}
	Let $L :V\rightarrow V$ be a normal bounded shift-preserving operator acting on the shift-invariant space $V$ with finite length $\ele(V)=\ell$. Then $V$ can be decomposed as an orthogonal sum 
	\begin{equation*}
		V=S(\psi_1)\oplus \dots \oplus S(\psi_{\ell}),
	\end{equation*} where each principal shift-invariant subspace $S(\psi_i)$ is reducing for $L$, each function $\psi_i$ is a Parseval frame generator of $S(\psi_i)$, and $\sigma(S(\psi_{i+1}))\subset \sigma(S(\psi_i))$ for every $1 \leq i \leq \ell$.
\end{cor}
\bdem
Consider the subspaces $V_1,\dots,V_{\ell}$ given by Theorem \ref{triangular}. Since these subspaces are invariant for $L$ and $L$ is normal, they are reducing for $L$ by the lemma above. 
Then, the subspaces $V_{j}\ominus V_{j-1} = S(\psi_{j}) $ for $2\leq j\leq \ell$ are reducing for $L$ and satisfy the desired properties. The same holds for $V_1 = S(\psi_1)$. 
\edem

This corollary achieves a diagonal form for normal shift-preserving operators, consistent with the notion of diagonalization defined in \cite{ACCP}. Indeed, the orthogonal decomposition into reducing principal shift-invariant subspaces yields, from the fibers perspective, a measurable diagonalization of the associated range operator. Each principal subspace corresponds to a one-dimensional invariant fiber space, where the range operator acts by multiplication with a measurable eigenvalue function. In the next subsection, we develop this connection and explain in more detail how this decomposition recovers the main spectral theorem of \cite{ACCP}.

\bigskip

\subsection{Diagonal form for normal shift-preserving operators}

Given a sequence $a=\{a(h)\}_{h\in H}\in\ell^1(H)$, we denote its Fourier transform as
$$\hat{a}(\w)=\sum_{h\in H }a(h)e^{-2\pi i \w \cdot h},\quad \w\in I.$$ 
This extends to $\ell^2(H)$ and it holds that  $a\in\ell^2(H)$ if and only if $\hat{a} \in L^2(I)$. Moreover, if \ $\hat{a}\in L^{\infty}(I)$, we will say that $a$ is of  \textit{bounded spectrum}.

\begin{fed}\label{def:Lambda_a}
	Given $a\in\ell^2(H)$ of bounded spectrum, the multiplication operator $$M_{\hat{a}}:L^2( I,\ell^2(\La))\to L^2( I,\ell^2(\La))$$ is defined by $M_{\hat{a}} f = \hat{a}f$. We denote as $\Lambda_a:L^2(\R^d)\to L^2(\R^d)$ the operator $$\Lambda_a:=\mathcal T^{-1} M_{\hat{a}} \mathcal T,$$ 
	which is well-defined and bounded. 
\end{fed}

\begin{pro}[{\cite[Proposition 3.2]{BCCHM}}]
	Let $a=\{a(\la)\}_{\la\in \La}\in\ell^2(\La)$ be of bounded spectrum. Let $$\mathfrak B  :=\{\varphi\in L^2(\R^d)\,:\, \{T_h\varphi\}_{h\in H} \text{ is a Bessel sequence}\},$$
	then $$\Lambda_a f =\sum_{h\in H}a(h)T_h f, \quad \forall f\in\mathfrak B, $$
	with convergence in $L^2(\R^d)$.
\end{pro}

Since $\mathfrak B$ is a dense set of $L^2(\R^d)$, given that the functions of compact support in $L^2(\R^d)$ belong to $\mathfrak B$ (see \cite[Proposition 9.3.4]{Chr2003}), an alternative definition for $\Lambda_a$ can be given as the continuous extension of the bounded operator $\tilde\Lambda_a:\mathfrak{B}\to L^2(\R^d)$, 
\begin{equation}\label{eq:tilde-Lambda_a}
	\tilde\Lambda_a f := \sum_{h\in H}a(h)T_h f.
\end{equation}

\begin{fed}\label{s-eigenval}
	Let $V \subset L^2(\R^d)$ be a shift-invariant space and $L:V\rightarrow V$ a bounded shift-preserving operator. Given $a\in \ell^{2}(H)$ a sequence of bounded spectrum, we say that $\Lambda_a$ is an $s$-eigenvalue of $L$ if
	\begin{equation*} 
		V_a := \ker\left(L - \Lambda_a\right)\neq\{0\}.
	\end{equation*}
\end{fed} 

The main result in \cite{ACCP} reads as follows.

\begin{teo}[{\cite[Theorem 6.16]{ACCP}}]\label{spectral}
	Let $L :V\rightarrow V$ be a normal bounded shift-preserving operator on a finitely generated shift-invariant space $V.$
	Then there exist $\ell \in \N$, and sequences of bounded spectrum $a_1,\dots,a_\ell$,
	such that 
	$$ V=V_{a_1}\oplus\dots\oplus V_{a_\ell}\quad \text{ and }\quad L=\sum_{j=1}^\ell \Lambda_{a_j} P_{V_{a_j}},$$
	where $P_S$ denotes the orthogonal projection onto a closed subspace $S.$
\end{teo}

We show that, using Corollary \ref{spectral-cor}, we can give an alternative proof of Theorem \ref{spectral}.

\begin{proof}[Alternative proof of Theorem \ref{spectral}]
	Let $\ell=\ele(V)$. Consider the principal subspaces $S(\psi_1),\dots,S(\psi_\ell)$ in the decomposition given by Corollary \ref{spectral-cor}. Since they are reducing for $L$, the operator $L$ induces bounded normal operators $L_j$ in each of them. As each subspace is principal, then the range operator $R_j$ associated to $L_j$ is a scalar. More precisely, for any $f\in S(\psi_j)$ 
	$$
	R_j(\w)(\ete f)(\w)=R(\w)(\ete f)(\w)=\lambda_j(\w) (\ete f)(\w),
	$$
	for some measurable function $\lambda_j: I \to\C$ and for almost every $\w\in I$. Let $a_j\in \ell^2(H)$ be such that $\widehat{a}_j=\lambda_j$, for every $1\leq j \leq \ell$. 
	Now, the result follows by taking
	$$
	\La_{a_j}=\sum_{h \in H} a_j(h) T_h  \peso{and} V_{a_j}= S(\psi_j).
	$$
\end{proof}

\section{Bases of exponentials with periodic set of frequencies in Paley-Wiener spaces with multi-tiling spectrum}\label{sec-4}

As mentioned in the introduction, the second goal of this note (P2) is to present a new characterization of the multi-tiling sets that admit a structured Riesz basis of exponentials. In contrast to the result obtained in \cite{CUC}, where the characterization was obtained in terms of the Bohr compactification of the lattice involved, our approach is formulated within the ambient space $\T^{k\times k}$, giving a much more tractable condition. In addition, we propose a simple geometric criterion that weakens the admissibility assumption and prove that any multi-tiling set satisfying this criterion admits a structured Riesz basis of exponentials.

\subsection{Paley-Wiener spaces with multi-tiling spectrum}

We focus on an important class of shift-invariant spaces, namely the Paley-Wiener spaces. Given a measurable subset $\Omega\subset \R^d$ of finite measure, the  \textit{Paley-Wiener space} $PW_{\Omega}$ is defined as 
$$PW_\Omega = \{f\in L^2(\R^d)\,:\, \supp \hat{f} \subset \Omega\}.$$
These spaces are invariant under all translations, in particular, under translations by the lattice $H$.  

In this work, we study the Paley-Wiener spaces whose spectrum has the geometric structure of a multi-tiling set. A measurable set $\Omega$ is a  \textit{multi-tiling} set, or more precisely  \textit{$k$-tiles} $\R^d$ by translations of the lattice $\La$  if
\begin{equation}\label{eq:k-tile}
	\sum_{\la\in \La} \chi_{\Omega} (\w-\la)=k,\quad  \text{a.e. } \w \in \R^d.
\end{equation}

This condition can be checked by restricting to the fundamental domain $I\subset \R^d$. Notice that if $\Omega$ is a disjoint union of $k$ sets, each being  a 1-tiling set, then the condition \eqref{eq:k-tile} is satisfied. In fact, the converse also holds, as established in the following lemma, which was first proven in \cite[Lemma 1]{Kol} for the Euclidean setting and later extended to the context of LCA groups in~\cite[Lemma 2.4]{AAC}.

\begin{lem}
	Let $\La\subset \R^d$ be  a countable discrete subgroup.
	A measurable set $\Omega\subset \R^d, \;k$-tiles $\R^d$ by  translations of $\La$,
	if and only if  $$\Omega = \Omega_1\cup\dots\cup \Omega_k\cup R,$$ where $R$ is a zero measure set, 
	and the sets $\Omega_j$,  $ 1\leq j\leq k$ are measurable, disjoint and each of them tiles $\R^d$ by translations of  $\La$.
\end{lem}

This lemma implies that, for almost every $\w\in I$, the set 
$$
\La_\w:=\{\la\in\La: \ \la+\w\in\W\}
$$
has cardinality $k$. For simplicity, we will assume that this holds for every $\w\in I$. As in \cite{CUC},  following the lexicographic order in $\R^d$, we associate to each set $\La_\w$ the vector 
$$
\vec{\la}(\w) := (\la_1(\w),\dots,\la_k(\w))\in(\R^d)^k,
$$ 
Moreover, given a set $J\subseteq I$, we define  
$$\La_\W(J) := \{\,\vec{\la}(\w)\,:\, \w\in J\, \} \subset (\R^d)^k.$$

Now, let $PW_\Omega$ be the Paley-Wiener space associated to a $k$-tiling subset $\Omega$ of $\R^d$. If $J_\Omega$ denotes the range function associated to $PW_\Omega$,  then the previous observations show that $\dim J_\W(\w)=k$ for every $\w\in I$. In other words, 
\begin{equation*}
	J_\Omega(\w)\simeq\C^k,
\end{equation*}
see, for example, \cite[Proposition 2.12]{CC}. This allows us to study the problem of characterizing Riesz basis generators for the shift-invariant space $PW_\Omega$ through a simple condition formulated in the finite-dimensional space $\C^k$.

\begin{teo}[{\cite[Theorem 2.13]{AAC}}]\label{ShiftBFR0}
	Let $\Omega$ be a $k$-tiling measurable subset of $\R^d$. Given $\phi_1,\ldots,\phi_k \in PW_\Omega$ we define
	$$
	T_\w=\begin{pmatrix}
		\widehat{\phi}_1(\w+\la_{1})&\ldots& \widehat{\phi}_k(\w+\la_{1})\\
		\vdots&\ddots&\vdots\\
		\widehat{\phi}_1(\w+\la_{k})&\ldots& \widehat{\phi}_k(\w+\la_{k})
	\end{pmatrix},
	$$ 
	where $\vec{\la}(\w) = (\la_1(\w), \dots, \la_k(\w))$. Then, the following  statements are equivalent:
	
	\begin{itemize}
		\item[i)] The set $\{T_h\phi_j:\ h\in H\,,\ j=1,\ldots,k\}$ is a Riesz basis for $PW_\Omega$.
		\item[ii)]  There exist $A, B>0$ such that for almost every $\omega \in I$,
		\begin{equation}\label{item-2}
			A||x||^2 \leq \|T_\w \,x\|^2\leq B||x||^2,
		\end{equation}
		for every $x\in \C^k$.
	\end{itemize}
	Moreover, in this case, the constants of the Riesz basis are 
	$$
	A=\inf_{\w\in I} \ \|T_\w^{-1}\|^{-1} \peso{and} B=\sup_{\w\in I}\ \|T_\w\|.
	$$
\end{teo}

The idea of connecting this result to the context of Riesz bases of exponentials was introduced in \cite{AAC}. For the sake of simplicity, we let $e_a(x):=e^{2\pi i a  x}$ for every $a,x\in\R^d$ and use this notation interchangeably until the end of this section. The key step consists in considering the case where the generators are reproducing kernels of $PW_\Omega$, i.e., the functions $\phi_j$ satisfy
$$\hat{\phi}_{j} = e_{a_j}\chi_\Omega,$$
for some $a_1,\dots,a_k\in \R^d$. In this setting, since $\mathcal{F} (T_h\phi_{j}) = e_he_{a_j}\chi_\Omega$, the set 
\begin{equation}\label{eq structured RB}
	\{e_{a_j +h}:\ h\in H\,,\ j=1,\ldots,k\}
\end{equation}
is a Riesz basis of $L^2(\Omega)$ if and only if  $\{T_h\phi_{a_j}:\ h\in H\,,\ j=1,\ldots,k\}$ is a Riesz basis of  $PW_\Omega$. A basis of exponentials with the form as in \eqref{eq structured RB} is called  \textit{structured Riesz basis}.

Hence, the existence of a structured Riesz basis of exponentials for $L^2(\Omega)$ is equivalent to finding a vector $\vec{a}= (a_1,\dots,a_k) \in (\R^d)^k$ for which there exist $A,B>0$ such that for a.e. $\w\in I$ the inequalities in \eqref{item-2} hold for every $x\in\C^k$. A closer look at the matrices $T_\w$ reveals that they can be decomposed as
\begin{align*}
	T_\w&=\begin{pmatrix}
		\widehat{\phi}_1(\w+\la_{1})&\ldots& \widehat{\phi}_k(\w+\la_{1})\\
		\vdots&\ddots&\vdots\\
		\widehat{\phi}_1(\w+\la_{k})&\ldots& \widehat{\phi}_k(\w+\la_{k})
	\end{pmatrix}=\begin{pmatrix}
		e_{a_1}\, (\w+\la_1)     &\ldots&       e_{a_k}\, (\w+\la_1)  \\
		\vdots&\ddots&\vdots\\
		e_{a_1}\, (\w+\la_k)       &\ldots&    e_{a_k}(\w+\la_k) 
	\end{pmatrix}\nonumber \\&=
	\begin{pmatrix}
		e_{a_1}(\la_1) &\ldots& e_{a_k}(\la_1)   \\
		\vdots&\ddots&\vdots\\
		e_{a_1}(\la_{k})&\ldots& e_{a_k}(\la_{k})
	\end{pmatrix}\begin{pmatrix}
		e_{a_1}\, (\w)   &0&\ldots&0& 0\\
		0&e_{a_2}\, (\w)  &\ldots& 0&0\\
		\vdots&\vdots&\ddots&\vdots&\vdots\\
		0&0&\ldots& e_{a_{k-1}}\, (\w) &0\\
		0&0&\ldots& 0&e_{a_k}\, (\w) 
	\end{pmatrix}\nonumber\\ & = E_\w\, U_\w\,, 
\end{align*}
where $U_\w$ is a unitary matrix. 
Thus, to check condition \eqref{item-2} it is enough to consider the matrices
\begin{align}
	E_\w=E_{\vec{a},\vec{\la}}&=
	\begin{pmatrix}
		e_{a_1}(\la_1) &\ldots& e_{a_k}(\la_1)   \\
		\vdots&\ddots&\vdots\\
		e_{a_1}(\la_{k})&\ldots& e_{a_k}(\la_{k})
	\end{pmatrix}\,\quad \text{where } \vec{\la}(\w) = (\la_1(\w),\dots,\la_k(\w)), \label{Matrix E}
\end{align}
and show that for a.e. $\w\in I$,
$$A||x||^2 \leq \|E_{\vec{a},\vec{\la}(\w)} \,x\|^2\leq B||x||^2$$ for every $x\in\C^k$. This proves the following characterization of those sets $\W$ such that the space $L^2(\W)$ admits a structured Riesz basis of exponentials, which was essentially proved in \cite{AAC} even for more general LCA groups. 

\begin{teo}\label{caracterizacion1}
	Let $\W\subset\R^d$ be a $k$-tiling set by translations of the lattice $\La$ and let $\vec{a} = (a_1,\dots,a_k)\in(\R^d)^k$. The following statements are equivalent: 
	\begin{enumerate}
		\item[i)] The set $\{\,e_{a_j + h}\,:\, h\in H,\, j=1,\dots, k\,\}$ is a structured Riesz basis of exponentials of $L^2(\W)$.
		\item[ii)] There exist constants $A,B>0$ such that the inequalities 
		\begin{equation}\label{equiv-inequ}
			A||x||^2 \leq \|E_{\vec{a},\vec{\la}(\w)} \,x\|^2\leq B||x||^2
		\end{equation}
		hold for every $x\in\C^k$ and almost every $\w\in I$.
	\end{enumerate}
\end{teo}

\subsection{Sufficient and necessary conditions for the existence of structured Riesz bases of exponentials}

To state our first result, let us introduce a new concept of separation.

\begin{fed}\label{def:sep}
	Given $\alpha \in\R^d$ and $\delta>0$, we say that a subset $\Gamma\subseteq (\R^d)^k$ is $(\alpha,\delta)$-separated if for every $\vec{\gamma}=(\gamma_1,\ldots,\gamma_k)\in\Gamma$ and $j\neq \ell$
	$$
	|e^{2\pi i  \alpha \gamma_{j}}-e^{2\pi i  \alpha \gamma_{\ell} }|\geq \delta.
	$$
\end{fed}

Using this terminology, we can prove the following theorem, which is inspired by \cite[Corollary 4.2]{CUC}.

\begin{teo}\label{separados}
	Let $\W$ be a multi-tiling set such that  the set $\La_\W(J)$ is $( \alpha,\delta)$-separated for some $ \alpha \in\R^d$, $\delta>0$, and a full measure subset $J\subseteq I$. Then, $L^2(\W)$ admits a structured Riesz basis of exponentials.
\end{teo}

We postpone the proof of this result to the end of this section. This theorem can be seen as a generalization of \cite[Theorem 1.1]{CC}. Indeed, following the terminology of \cite{CC},  we say that $\W$ is  \textit{admissible} for $\La$ if there exists $v\in \R^d$ and a number $n\in\N$ such that for almost every $\w\in I$ 
\begin{equation}\label{admissible}
	v\cdot \la_1(\w),\ldots,v\cdot \la_k(\w)
\end{equation}
are pairwise distinct integer numbers ($\mbox{mod}\ n$), where $\vec{\la}(\w) = (\la_1(\w),\dots,\la_k(\w))$. Note that in this case, $\La_\W(J)$ is $(v/n,\delta)$-separated, for $\delta=|1-e^{2\pi i \frac1n}|$ and some full measure set $J\subseteq I$. As a consequence of Theorem \ref{separados}, we recover the main result from \cite{CC}.

\begin{teo}
	Let $\W$ be a multi-tiling set that is admissible for $\La$. Then, $L^2(\W)$ admits a structured Riesz basis of exponentials.
\end{teo}

It is known that the admissibility is not a necessary condition for a multi-tiling set $\Omega$ to admit a structured Riesz basis of exponentials (see \cite[Example 3.3]{CC}). Remarkably, the $(\alpha,\delta)$-separation is, nevertheless, a necessary condition for $2$-tiling sets, as we show in the following theorem.

\begin{teo}\label{necesario para 2tilings}
	Let $\W$ be a 2-tiling set such that $L^2(\W)$ admits a structured Riesz basis of exponentials. Then, there exists a full measure subset $J\subseteq I$, $ \alpha \in \R^d$ and $\delta>0$ such that $\La_\W(J)$ is $( \alpha,\delta)$-separated.
\end{teo}
\bdem
By Theorem \ref{caracterizacion1}, we know that there exists $\vec{a}=(a_1,a_2)\in \R^d\times\R^d$, and a full measure subset $J\subseteq I$ such that the matrices 
$$
E_{\vec{a},\vec{\la}(\w)}=\begin{pmatrix}
	e^{2\pi i a_1  \la_1} & e^{2\pi i a_2 \la_1}  \\
	
	e^{2\pi i a_1 \la_2}& e^{2\pi i a_2 \la_2}
\end{pmatrix}
$$
satisfy \eqref{equiv-inequ} for every $\w\in J$. These matrices can be factorized as
$$
E_{\vec{a},\vec{\la}(\w)}=\begin{pmatrix}
	e^{2\pi i a_1  \la_1} & 0  \\
	
	0& e^{2\pi i a_2 \la_2}
\end{pmatrix}
\begin{pmatrix}
	1& e^{2\pi i (a_2-a_1) \la_1}  \\			
	e^{2\pi i (a_1-a_2) \la_2}& 1
\end{pmatrix}
=:V_{\vec{a},\vec{\la}(\w)}  \widetilde{E}_{\vec{a},\vec{\la}(\w)}.
$$
Since the matrices $V_{\vec{a},\vec{\la}(\w)}$ are unitary, the matrices $ \widetilde{E}_{\vec{a},\vec{\la}(\w)}$ also satisfy \eqref{equiv-inequ}  uniformly in $J$. In particular, since $A$ is a lower bound for the smallest eigenvalue of the matrix  $|\widetilde{E}_{\vec{a},\vec{\la}(\w)}|^2$, we get that
$$
\det\left(\left| \widetilde{E}_{\vec{a},\vec{\la}(\w)}\right|^2\right)\geq A^2>0,
$$
also uniformly in $J$. By the properties of the determinant
$$
\det\left(\left| \widetilde{E}_{\vec{a},\vec{\la}(\w)}\right|^2\right)=\left|\det\left( \widetilde{E}_{\vec{a},\vec{\la}(\w)}\right)\right|^2=\left|1-e^{2\pi i (a_2-a_1)(\la_1-\la_2)}\right|^2.
$$
Hence, taking $ \alpha=a_2-a_1$, we obtain that 
$$
\left|e^{2\pi i  \alpha  \la_2}-e^{2\pi i  \alpha  \la_1}\right|\geq A,
$$
uniformly for every $\w\in J$. As a consequence, $\La_\W(J)$ is $( \alpha,A)$-separated.
\edem

We now provide another condition for $( \alpha,\delta)$-separation for a subset $\Gamma\subseteq (\R^d)^k$, which we state in a straightforward lemma. For this purpose, given $ \alpha\in\R^d$, we define the following subset of $\T$:
\begin{equation}\label{eq diferencias}
	\Delta_{\alpha}(\Gamma):=\{e^{2\pi i  \alpha (\gamma_j-\gamma_\ell)}:\ j\neq \ell,\ \vec\gamma\in\Gamma\}. 
\end{equation}
It is important to observe that the differences $\gamma_j-\gamma_\ell$ are between two different entries of some $\vec\gamma\in\Gamma$. The entries of two different elements of $\Gamma$ are not compared.

\begin{lem}\label{lem:separated}
	A set  $\Gamma\subseteq (\R^d)^k$ is $(\alpha,\delta)$-separated if and only if $1\notin \overline{\Delta_{ \alpha}(\Gamma)}$, where the closure is taken in $\T$. 
\end{lem}

The condition in Lemma \ref{lem:separated} is more closely related to the characterization given in \cite{CUC}. In that work, the authors provided a sufficient and necessary condition on a $k$-tiling set $\W$ to ensure that $L^2(\W)$ admits a structured Riesz basis of exponentials. This condition involves the Bohr compactification of $\La^k$, denoted by $\overline{\La^k}$. More precisely, it is formulated in terms of a lifting of the matrix $E_{\vec{a},\vec{\la}}$ to a matrix function defined on the closure of $\La_\W(I)$ in the Bohr topology of $\overline{\La^k}$. As it is usual, we are identifying $\La_\W(I)$ with its image inside $\overline{\La^k}$ by the canonical inclusion. 

To conclude this work, we will show an alternative sufficient and necessary condition on a $k$-tiling set $\W$ such that $L^2(\W)$ admits a structured Riesz basis of exponentials. Our condition is closer in spirit to the $(\alpha,\delta)$-separation or the admissibility. We begin with the following definition.

\begin{fed}
	Given $\vec{a}=(a_1,\dots,a_k)\in(\R^d)^k$, consider the function $\fii_{\vec{a}}:\La^k\to\T^{k\times k}$ defined by
	$$
	\fii_{\vec{a}}(\vec{\la})=
	\begin{pmatrix}
		e^{2\pi i a_1 \la_1} &\ldots& e^{2\pi i a_k  \la_1}  \\
		\vdots&\ddots&\vdots\\
		e^{2\pi i a_1 \la_k}&\ldots& e^{2\pi i a_k  \la_k}
	\end{pmatrix}		
	$$
\end{fed}	

Clearly, $\T^{k\times k}$ is a compact subset of matrices, which can be thought of as a locally compact abelian group when endowed with the entrywise product. In that case, $\fii_{\vec{a}}$ becomes a group morphism. Since the determinant acts continuously in this set, we get the following alternative way to describe those $k$-tiling sets $\W$ for which $L^2(\W)$ admits a structured Riesz basis of exponentials. 

Moreover, analogously as in \cite{CUC}, we define the following two auxiliary sets. For a vector $\vec{v}\in\La^k$ we denote
$$I_{\vec{v}}:= \left\{ \w\in I \,:\, \vec{\la}(\w) = \vec{v}\right\},$$
and 
$$Q := \bigcup \left\{I_{\vec{v}}\,: \, \vec{v}\in \La^k \text{ such that } |\,I_{\vec{v}}\,|>0\right\}\subseteq I.$$
It is easy to check that $|I\setminus Q|=0$.

\begin{teo}\label{nueva caracterizacion}
	Let $\W\subset\R^d$ be a $k$-tiling set by translations of the lattice $\La$ and let $\vec{a}=(a_1,\dots,a_k)\in(\R^d)^k$. Then the following statements are equivalent: 
	\begin{enumerate}
		\item[i)] The set $\{\,e_{a_j + h}\,:\, h\in H,\, j=1,\dots, k\,\}$ is a structured Riesz basis of exponentials of $L^2(\W)$.
		\item[ii)] The determinant is different from zero in $\overline{\fii_{\vec{a}}(\La_\W(J))}$, where the closure is taken in $\T^{k\times k}$ and $J\subseteq Q$ is a full-measure subset.
	\end{enumerate}
\end{teo}	
\begin{proof}
	First of all, note that the matrices in $\overline{\fii_{\vec{a}}(\La_\W(J))}$ have Frobenius norm exactly equal to $k$. Hence, they are also uniformly bounded in the spectral norm. Since the matrices 
	$E_{\vec{a},\vec{\la}}$ belong to this set, the upper bound in \eqref{equiv-inequ} always hold. Let $B$ denote this uniform upper bound.

	We now assume that statement i) holds. By Theorem \ref{caracterizacion1}, this implies that the lower bound in \eqref{equiv-inequ} holds for almost every $\w\in I$ and for every $x\in \C^k$. Let us assume that there exists a matrix $M\in \overline{\fii_{\vec{a}}(\La_\W(J))}$ such that $\det(M)=0$, and show that it leads to a contradiction. The existence of such $M$ implies that there exists a sequence $\{M_n\}_{n\in\N}$ with $M_n\in \fii_{\vec{a}}(\La_\W(J))$ and  $M_n \to M$. Each $M_n$ can be written as $M_n=E_{\vec{a},\vec{\la}_n}$, for some sequence $\{\vec\la_n\}_{n\in\N}$ in $\La_\W(J)$. 
	Using standard properties of the determinant, we get:
	$$
	\det\left(\left|E_{\vec{a},\vec{\la}_n}\right|^2\right)=\left|\det\left(E_{\vec{a},\vec{\la}_n}\right) \right|^2\xrightarrow[n\to\infty]{} \left|\det(M)\right|^2=0. 
	$$
	Let $\mu_{\vec{\la}_n}$ denote the smallest eigenvalue of $|E_{\vec{a},\vec{\la}_n}|^2$. From the above, if follows that $\mu_{\vec{\la}_n}\to 0$. Choose $n\in\N$ such that $\mu_{\vec{\la}_n}<A$, and let $x\in \C^k$ be an eigenvalue of $|E_{\vec{a},\vec{\la}_n}|^2$ associated to $\mu_{\vec{\la}_n}$. We have that 
	$$ \left\| E_{\vec{a},\vec{\la}_n} x \right\|^2=\mu_{\vec{\la}_n}\|x\|^2<A\|x\|^2.$$
	Since $\vec{\la}_n\in \La_\W(J)$ and $J\subseteq Q$, we conclude that there exists a set of positive measure, namely $I_{\vec{\la}_n}$, on which the lower bound in \eqref{equiv-inequ} fails. This gives a contradiction.
	
	Conversely, let us assume that statement ii) holds. Hence, there exists $\delta>0$ such that
	$
	|\det(M)|^2\geq \delta,
	$
	for every matrix $M\in \overline{\fii_{\vec{a}}(\La_\W(J))}$. Given any $\w\in J$, let $\mu_{\vec{\la}(\w)}$ denote the smallest eigenvalue of the matrix $|E_{\vec{a},\vec{\la}(\w)}|^2$. Then we obtain that
	$$
	\delta\leq \det\left(\left|E_{\vec{a},\vec{\la}(\w)}\right|^2\right) \leq \mu_{\vec{\la}(\w)} B^{k-1}.
	$$
	Therefore, it follows that $\mu_{\vec{\la}(\w)}\geq A$ where $A=\delta/B^{(k-1)}$, for every $\w\in J$. This implies that the left-hand side inequality in \eqref{equiv-inequ} holds for a.e. $\w\in I$ and every $x\in \C^k$, as we wanted to show.
\end{proof}

In comparison with the main result of \cite{CUC}, this theorem provides a more tractable condition, since the closure is taken in $\T^{k\times k}$ instead of in the Bohr compactification of $\La$.  Another difference is the following. In Theorems 3.3  and 4.1 of \cite{CUC}, the embedding of $\La_\W(Q)$ on $\overline{\La}$ is fixed, and the different vectors $\vec{a}\in (\R^d)^k$ lead to different (continuous) functions that have to be tested on the closure of $\La_\W(Q)$ inside $\overline{\La}$. In our result, for different vectors $\vec{a}\in (\R^d)^k$ we have different morphisms of 
$\La_\W(Q)$ inside $\T^{k\times k}$, but the continuous function that has to be different from zero in the closure is fixed. 

To conclude this section, we use Theorem \ref{nueva caracterizacion}, to provide a simple proof of Theorem \ref{separados}.

\bdem[Proof of Theorem \ref{separados}]
Assume that the set $\La_\W(J)$ is $(\alpha,\delta)$-separated for some $\alpha \in\R^d$, $\delta>0$, and a full measure subset $J\subseteq I$. We may assume that $J\subseteq Q$, otherwise we replace $J$ with $J\cap Q$ which is also a full measure subset of $I$. Consider the vector $\vec{a}=(\alpha,2\alpha,\ldots,k \alpha)\in (\R^d)^k$. Then, for any $\vec{\la}\in \La_\W(J)$ we have that
$$
\fii_{\vec{a}}(\vec{\la})=
\begin{pmatrix}
	e^{2\pi i \alpha \la_1} &\ldots& e^{2\pi i (\alpha  \la_1) k }  \\
	\vdots&\ddots&\vdots\\
	e^{2\pi i \alpha  \la_k}&\ldots& e^{2\pi i (\alpha\la_k) k }
\end{pmatrix},	
$$
which is a Vandermonde matrix. In particular, by the hypothesis on  $(\alpha,\delta)$-separation of $\La_\W(J)$, we obtain
$$
\left|\det(\fii_{\vec{a}}(\vec{\la}))\right|^2=\prod_{j\neq \ell} \left|e_\alpha(\la_j)-e_\alpha(\la_\ell)\right|\geq \delta^{k(k-1)}>0.
$$
Since the determinant is continuous as a function acting on the group of $k\times k$ matrices, we get that the determinant is different from zero in $\overline{\fii_{\vec{a}}(\La_\W(J))}$. 
The proof is then concluded by Theorem \ref{nueva caracterizacion}.
\edem

\section*{Acknowledgments}

\noindent E.A.  was supported by Grants PID2022-142202NB-I00, funded by \\ MICIU/AEIMCIN/AEI/10.13039/501100011033 and by\\ /10.13039/501100011033FEDER, UE,  PIP 11220210100954CO.

\noindent J.A. was supported by Grants PID2024-160033NB-I00 funded by \\ MICIU/AEIMCIN/AEI/10.13039/501100011033 and by\\ /10.13039/501100011033FEDER, UE,  PIP 11220210100954CO, and UNLP11X829. 

\noindent D.C. was supported by the European Union’s programme Horizon Europe, HORIZON MSCA 2021 PF 01, Grant agreement No. 101064206.

\printbibliography

@article{AC,
	AUTHOR = {Anastasio, Magal\'{\i} and Cabrelli, Carlos},
	TITLE = {Sampling in a union of frame generated subspaces},
	JOURNAL = {Sampl. Theory Signal Image Process.},
	FJOURNAL = {Sampling Theory in Signal and Image Processing. An
	International Journal},
	VOLUME = {8},
	YEAR = {2009},
	NUMBER = {3},
	PAGES = {261--286},
	MRCLASS = {94A20 (94A08 94A12)},
	MRNUMBER = {2590917},
	MRREVIEWER = {Tian Zhou Xu},
}

@article{AAC,
	AUTHOR = {Agora, Elona and Antezana, Jorge and Cabrelli, Carlos},
	TITLE = {Multi-tiling sets, {R}iesz bases, and sampling near the
	critical density in {LCA} groups},
	JOURNAL = {Adv. Math.},
	FJOURNAL = {Advances in Mathematics},
	VOLUME = {285},
	YEAR = {2015},
	PAGES = {454--477},
	ISSN = {0001-8708},
	MRCLASS = {94A20 (22B05)},
	MRNUMBER = {3406506},
	MRREVIEWER = {Yunus Emre Yildirir},
	DOI = {10.1016/j.aim.2015.08.006},
	URL = {https://doi.org/10.1016/j.aim.2015.08.006},
}

@article{A,
	AUTHOR = {Azoff, Edward A.},
	TITLE = {Borel measurability in linear algebra},
	JOURNAL = {Proc. Amer. Math. Soc.},
	FJOURNAL = {Proceedings of the American Mathematical Society},
	VOLUME = {42},
	YEAR = {1974},
	PAGES = {346--350},
	ISSN = {0002-9939},
	MRCLASS = {15A60 (47C05)},
	MRNUMBER = {327799},
	MRREVIEWER = {T. Rolf Turner},
	DOI = {10.2307/2039503},
	URL = {https://doi.org/10.2307/2039503},
}

@article{B,
	AUTHOR = {Bownik, Marcin},
	TITLE = {The structure of shift-invariant subspaces of {$L^2({\bf
	R}^n)$}},
	JOURNAL = {J. Funct. Anal.},
	FJOURNAL = {Journal of Functional Analysis},
	VOLUME = {177},
	YEAR = {2000},
	NUMBER = {2},
	PAGES = {282--309},
	ISSN = {0022-1236},
	MRCLASS = {42C15 (46E30 47B38)},
	MRNUMBER = {1795633},
	MRREVIEWER = {Ole Christensen},
	DOI = {10.1006/jfan.2000.3635},
	URL = {https://doi.org/10.1006/jfan.2000.3635},
}

@article{BDR2,
	AUTHOR = {de Boor, Carl and DeVore, Ronald A. and Ron, Amos},
	TITLE = {The structure of finitely generated shift-invariant spaces in
	{$L_2({\bf R}^d)$}},
	JOURNAL = {J. Funct. Anal.},
	FJOURNAL = {Journal of Functional Analysis},
	VOLUME = {119},
	YEAR = {1994},
	NUMBER = {1},
	PAGES = {37--78},
	ISSN = {0022-1236},
	MRCLASS = {46E30 (41A63 42B99 47B38)},
	MRNUMBER = {1255273},
	MRREVIEWER = {Rong Qing Jia},
	DOI = {10.1006/jfan.1994.1003},
	URL = {https://doi.org/10.1006/jfan.1994.1003},
}

@article{BCCHM,
	AUTHOR = {Barbieri, D. and Cabrelli, C. and Carbajal, D. and Hern\'{a}ndez,
	E. and Molter, U.},
	TITLE = {The structure of group preserving operators},
	JOURNAL = {Sampl. Theory Signal Process. Data Anal.},
	FJOURNAL = {Sampling Theory, Signal Processing, and Data Analysis},
	VOLUME = {19},
	YEAR = {2021},
	NUMBER = {1},
	PAGES = {Paper No. 5, 28},
	ISSN = {2730-5716},
	MRCLASS = {43A25 (47A15 47B15)},
	MRNUMBER = {4301192},
	MRREVIEWER = {Serap \"{O}ztop},
	DOI = {10.1007/s43670-021-00005-3},
	URL = {https://doi.org/10.1007/s43670-021-00005-3},
}

@article{BI2019,
	AUTHOR = {Bownik, Marcin and Iverson, Joseph W.},
	TITLE = {Multiplication-invariant operators and the classification of
	{LCA} group frames},
	JOURNAL = {J. Funct. Anal.},
	FJOURNAL = {Journal of Functional Analysis},
	VOLUME = {280},
	YEAR = {2021},
	NUMBER = {2},
	PAGES = {Paper No. 108780, 59},
	ISSN = {0022-1236},
	MRCLASS = {42C15 (43A32 43A65 46C05 47A15)},
	MRNUMBER = {4159269},
	MRREVIEWER = {Judith A. Packer},
	DOI = {10.1016/j.jfa.2020.108780},
	URL = {https://doi.org/10.1016/j.jfa.2020.108780},
}

@article{CC,
	AUTHOR = {Cabrelli, Carlos and Carbajal, Diana},
	TITLE = {Riesz bases of exponentials on unbounded multi-tiles},
	JOURNAL = {Proc. Amer. Math. Soc.},
	FJOURNAL = {Proceedings of the American Mathematical Society},
	VOLUME = {146},
	YEAR = {2018},
	NUMBER = {5},
	PAGES = {1991--2004},
	ISSN = {0002-9939},
	MRCLASS = {42B10 (42A10 42A15 42C15)},
	MRNUMBER = {3767351},
	MRREVIEWER = {Tuba Ari},
	DOI = {10.1090/proc/13980},
	URL = {https://doi.org/10.1090/proc/13980},
}

@article{CP,
	AUTHOR = {Cabrelli, Carlos and Paternostro, Victoria},
	TITLE = {Shift-modulation invariant spaces on {LCA} groups},
	JOURNAL = {Studia Math.},
	FJOURNAL = {Studia Mathematica},
	VOLUME = {211},
	YEAR = {2012},
	NUMBER = {1},
	PAGES = {1--19},
	ISSN = {0039-3223},
	MRCLASS = {43A77 (43A15)},
	MRNUMBER = {2990556},
	MRREVIEWER = {Alexander Isaakovich Shtern},
	DOI = {10.4064/sm211-1-1},
	URL = {https://doi.org/10.4064/sm211-1-1},
}

@book{Chr2003,
	AUTHOR = {Christensen, Ole},
	TITLE = {An introduction to frames and {R}iesz bases},
	SERIES = {Applied and Numerical Harmonic Analysis},
	EDITION = {Second},
	PUBLISHER = {Birkh\"{a}user/Springer, [Cham]},
	YEAR = {2016},
	PAGES = {xxv+704},
	ISBN = {978-3-319-25611-5; 978-3-319-25613-9},
	MRCLASS = {42-02 (42C15 42C40 46B15 46C05)},
	MRNUMBER = {3495345},
	MRREVIEWER = {Marcin M. Bownik},
	DOI = {10.1007/978-3-319-25613-9},
	URL = {https://doi.org/10.1007/978-3-319-25613-9},
}

@article{ACCP,
	AUTHOR = {Aguilera, A. and Cabrelli, C. and Carbajal, D. and
	Paternostro, V.},
	TITLE = {Diagonalization of shift-preserving operators},
	JOURNAL = {Adv. Math.},
	FJOURNAL = {Advances in Mathematics},
	VOLUME = {389},
	YEAR = {2021},
	PAGES = {Paper No. 107892, 32},
	ISSN = {0001-8708},
	MRCLASS = {47A15 (42C15 47A05 94A20)},
	MRNUMBER = {4288218},
	MRREVIEWER = {Arkady K. Kitover},
	DOI = {10.1016/j.aim.2021.107892},
	URL = {https://doi.org/10.1016/j.aim.2021.107892},
}

@article{CUC,
	AUTHOR = {Cabrelli, Carlos and Hare, Kathryn E. and Molter, Ursula},
	TITLE = {Riesz bases of exponentials and the {B}ohr topology},
	JOURNAL = {Proc. Amer. Math. Soc.},
	FJOURNAL = {Proceedings of the American Mathematical Society},
	VOLUME = {149},
	YEAR = {2021},
	NUMBER = {5},
	PAGES = {2121--2131},
	ISSN = {0002-9939},
	MRCLASS = {42B99 (05B45 42A10 42A15 42C15)},
	MRNUMBER = {4232203},
	MRREVIEWER = {Alfredo Lazaro Gonz\'{a}lez},
	DOI = {10.1090/proc/15395},
	URL = {https://doi.org/10.1090/proc/15395},
}

@article{DL22,
	AUTHOR = {Debernardi, Alberto and Lev, Nir},
	TITLE = {Riesz bases of exponentials for convex polytopes with
	symmetric faces},
	JOURNAL = {J. Eur. Math. Soc. (JEMS)},
	FJOURNAL = {Journal of the European Mathematical Society (JEMS)},
	VOLUME = {24},
	YEAR = {2022},
	NUMBER = {8},
	PAGES = {3017--3029},
	ISSN = {1435-9855},
	MRCLASS = {42C15 (52B11 94A20)},
	MRNUMBER = {4416595},
	MRREVIEWER = {S. Sivananthan},
	DOI = {10.4171/jems/1158},
	URL = {https://doi.org/10.4171/jems/1158},
}

@article{FMM06,
	AUTHOR = {Farkas, B\'{a}lint and Matolcsi, M\'{a}t\'{e} and M\'{o}ra, P\'{e}ter},
	TITLE = {On {F}uglede's conjecture and the existence of universal
	spectra},
	JOURNAL = {J. Fourier Anal. Appl.},
	FJOURNAL = {The Journal of Fourier Analysis and Applications},
	VOLUME = {12},
	YEAR = {2006},
	NUMBER = {5},
	PAGES = {483--494},
	ISSN = {1069-5869},
	MRCLASS = {52C22 (20K01 42B10)},
	MRNUMBER = {2267631},
	DOI = {10.1007/s00041-005-5069-7},
	URL = {https://doi.org/10.1007/s00041-005-5069-7},
}

@article{FR06,
	AUTHOR = {Farkas, B\'{a}lint and R\'{e}v\'{e}sz, Szil\'{a}rd Gy.},
	TITLE = {Tiles with no spectra in dimension 4},
	JOURNAL = {Math. Scand.},
	FJOURNAL = {Mathematica Scandinavica},
	VOLUME = {98},
	YEAR = {2006},
	NUMBER = {1},
	PAGES = {44--52},
	ISSN = {0025-5521},
	MRCLASS = {52C20 (42C15)},
	MRNUMBER = {2221543},
	MRREVIEWER = {Aicke Hinrichs},
	DOI = {10.7146/math.scand.a-14982},
	URL = {https://doi.org/10.7146/math.scand.a-14982},
}

@article{Fu01,
	AUTHOR = {Fuglede, Bent},
	TITLE = {Orthogonal exponentials on the ball},
	JOURNAL = {Expo. Math.},
	FJOURNAL = {Expositiones Mathematicae},
	VOLUME = {19},
	YEAR = {2001},
	NUMBER = {3},
	PAGES = {267--272},
	ISSN = {0723-0869},
	MRCLASS = {42C15},
	MRNUMBER = {1852076},
	MRREVIEWER = {Steen Pedersen},
	DOI = {10.1016/S0723-0869(01)80005-0},
	URL = {https://doi.org/10.1016/S0723-0869(01)80005-0},
}

@article{Fu74,
	AUTHOR = {Fuglede, Bent},
	TITLE = {Commuting self-adjoint partial differential operators and a
	group theoretic problem},
	JOURNAL = {J. Functional Analysis},
	FJOURNAL = {Journal of Functional Analysis},
	VOLUME = {16},
	YEAR = {1974},
	PAGES = {101--121},
	ISSN = {0022-1236},
	MRCLASS = {47F05 (81.47)},
	MRNUMBER = {470754},
	DOI = {10.1016/0022-1236(74)90072-x},
	URL = {https://doi.org/10.1016/0022-1236(74)90072-x},
}

@article{GL14,
	AUTHOR = {Grepstad, Sigrid and Lev, Nir},
	TITLE = {Multi-tiling and {R}iesz bases},
	JOURNAL = {Adv. Math.},
	FJOURNAL = {Advances in Mathematics},
	VOLUME = {252},
	YEAR = {2014},
	PAGES = {1--6},
	ISSN = {0001-8708},
	MRCLASS = {52C23 (42B10)},
	MRNUMBER = {3144222},
	MRREVIEWER = {Jordi Marzo},
	DOI = {10.1016/j.aim.2013.10.019},
	URL = {https://doi.org/10.1016/j.aim.2013.10.019},
}

@article{GL18,
	AUTHOR = {Grepstad, Sigrid and Lev, Nir},
	TITLE = {Riesz bases, {M}eyer's quasicrystals, and bounded remainder
	sets},
	JOURNAL = {Trans. Amer. Math. Soc.},
	FJOURNAL = {Transactions of the American Mathematical Society},
	VOLUME = {370},
	YEAR = {2018},
	NUMBER = {6},
	PAGES = {4273--4298},
	ISSN = {0002-9947},
	MRCLASS = {42C15 (11K38 52C23)},
	MRNUMBER = {3811528},
	MRREVIEWER = {Mihail N. Kolountzakis},
	DOI = {10.1090/tran/7157},
	URL = {https://doi.org/10.1090/tran/7157},
}

@book{He64,
	AUTHOR = {Helson, Henry},
	TITLE = {Lectures on invariant subspaces},
	PUBLISHER = {Academic Press, New York-London},
	YEAR = {1964},
	PAGES = {xi+130},
	MRCLASS = {30.85 (47.35)},
	MRNUMBER = {171178},
	MRREVIEWER = {Ronald G. Douglas},
}

@article{IKT03,
	AUTHOR = {Iosevich, Alex and Katz, Nets and Tao, Terence},
	TITLE = {The {F}uglede spectral conjecture holds for convex planar
	domains},
	JOURNAL = {Math. Res. Lett.},
	FJOURNAL = {Mathematical Research Letters},
	VOLUME = {10},
	YEAR = {2003},
	NUMBER = {5-6},
	PAGES = {559--569},
	ISSN = {1073-2780},
	MRCLASS = {42B99 (52A99 52C99)},
	MRNUMBER = {2024715},
	MRREVIEWER = {Andrei K. Lerner},
	DOI = {10.4310/MRL.2003.v10.n5.a1},
	URL = {https://doi.org/10.4310/MRL.2003.v10.n5.a1},
}

@article{Ko00,
	AUTHOR = {Kolountzakis, Mihail N.},
	TITLE = {Non-symmetric convex domains have no basis of exponentials},
	JOURNAL = {Illinois J. Math.},
	FJOURNAL = {Illinois Journal of Mathematics},
	VOLUME = {44},
	YEAR = {2000},
	NUMBER = {3},
	PAGES = {542--550},
	ISSN = {0019-2082},
	MRCLASS = {52C22 (41A65 42B05 46E30)},
	MRNUMBER = {1772427},
	MRREVIEWER = {B\'{e}la Uhrin},
	URL = {http://projecteuclid.org/euclid.ijm/1256060414},
}

@article{KM06,
	AUTHOR = {Kolountzakis, Mihail N. and Matolcsi, M\'{a}t\'{e}},
	TITLE = {Tiles with no spectra},
	JOURNAL = {Forum Math.},
	FJOURNAL = {Forum Mathematicum},
	VOLUME = {18},
	YEAR = {2006},
	NUMBER = {3},
	PAGES = {519--528},
	ISSN = {0933-7741},
	MRCLASS = {20K01 (52C22)},
	MRNUMBER = {2237932},
	MRREVIEWER = {Richard Kenyon},
	DOI = {10.1515/FORUM.2006.026},
	URL = {https://doi.org/10.1515/FORUM.2006.026},
}

@article{Kol,
	AUTHOR = {Kolountzakis, Mihail N.},
	TITLE = {Multiple lattice tiles and {R}iesz bases of exponentials},
	JOURNAL = {Proc. Amer. Math. Soc.},
	FJOURNAL = {Proceedings of the American Mathematical Society},
	VOLUME = {143},
	YEAR = {2015},
	NUMBER = {2},
	PAGES = {741--747},
	ISSN = {0002-9939},
	MRCLASS = {42B08},
	MRNUMBER = {3283660},
	MRREVIEWER = {Hikmet S. \"{O}zarslan},
	DOI = {10.1090/S0002-9939-2014-12310-0},
	URL = {https://doi.org/10.1090/S0002-9939-2014-12310-0},
}

@article{KN15,
	AUTHOR = {Kozma, Gady and Nitzan, Shahaf},
	TITLE = {Combining {R}iesz bases},
	JOURNAL = {Invent. Math.},
	FJOURNAL = {Inventiones Mathematicae},
	VOLUME = {199},
	YEAR = {2015},
	NUMBER = {1},
	PAGES = {267--285},
	ISSN = {0020-9910},
	MRCLASS = {46A35 (42C15)},
	MRNUMBER = {3294962},
	MRREVIEWER = {Zygmunt Wronicz},
	DOI = {10.1007/s00222-014-0522-3},
	URL = {https://doi.org/10.1007/s00222-014-0522-3},
}

@article{KN16,
	AUTHOR = {Kozma, Gady and Nitzan, Shahaf},
	TITLE = {Combining {R}iesz bases in {$\Bbb{R}^d$}},
	JOURNAL = {Rev. Mat. Iberoam.},
	FJOURNAL = {Revista Matem\'{a}tica Iberoamericana},
	VOLUME = {32},
	YEAR = {2016},
	NUMBER = {4},
	PAGES = {1393--1406},
	ISSN = {0213-2230},
	MRCLASS = {42C15 (42C30)},
	MRNUMBER = {3593529},
	MRREVIEWER = {R. A. Zalik},
	DOI = {10.4171/RMI/922},
	URL = {https://doi.org/10.4171/RMI/922},
}

@article{KNO23,
	AUTHOR = {Kozma, Gady and Nitzan, Shahaf and Olevski\v{i}, Alexander},
	TITLE = {A set with no {R}iesz basis of exponentials},
	JOURNAL = {Rev. Mat. Iberoam.},
	FJOURNAL = {Revista Matem\'{a}tica Iberoamericana},
	VOLUME = {39},
	YEAR = {2023},
	NUMBER = {6},
	PAGES = {2007--2016},
	ISSN = {0213-2230},
	MRCLASS = {42C15},
	MRNUMBER = {4671427},
	MRREVIEWER = {Chun-Kit Lai},
	DOI = {10.4171/rmi/1411},
	URL = {https://doi.org/10.4171/rmi/1411},
}

@article{La01,
	AUTHOR = {{\L}aba, I.},
	TITLE = {Fuglede's conjecture for a union of two intervals},
	JOURNAL = {Proc. Amer. Math. Soc.},
	FJOURNAL = {Proceedings of the American Mathematical Society},
	VOLUME = {129},
	YEAR = {2001},
	NUMBER = {10},
	PAGES = {2965--2972},
	ISSN = {0002-9939},
	MRCLASS = {42A99},
	MRNUMBER = {1840101},
	MRREVIEWER = {Steen Pedersen},
	DOI = {10.1090/S0002-9939-01-06035-X},
	URL = {https://doi.org/10.1090/S0002-9939-01-06035-X},
}

@article{LM22,
	AUTHOR = {Lev, Nir and Matolcsi, M\'{a}t\'{e}},
	TITLE = {The {F}uglede conjecture for convex domains is true in all
	dimensions},
	JOURNAL = {Acta Math.},
	FJOURNAL = {Acta Mathematica},
	VOLUME = {228},
	YEAR = {2022},
	NUMBER = {2},
	PAGES = {385--420},
	ISSN = {0001-5962},
	MRCLASS = {52A05 (28A75 42B10)},
	MRNUMBER = {4448683},
	MRREVIEWER = {Grigory M. Ivanov},
	DOI = {10.4310/acta.2022.v228.n2.a3},
	URL = {https://doi.org/10.4310/acta.2022.v228.n2.a3},
}

@article{Ma06,
	author    = {Marzo, Jordy},
	title     = {Riesz basis of exponentials for a union of cubes in $\mathbb{R}^d$},
	journal      = {arXiv:math/0601288},
	year      = {2006}
}

@article{RS,
	AUTHOR = {Ron, Amos and Shen, Zuowei},
	TITLE = {Frames and stable bases for shift-invariant subspaces of
	{$L_2(\bold R^d)$}},
	JOURNAL = {Canad. J. Math.},
	FJOURNAL = {Canadian Journal of Mathematics. Journal Canadien de
	Math\'{e}matiques},
	VOLUME = {47},
	YEAR = {1995},
	NUMBER = {5},
	PAGES = {1051--1094},
	ISSN = {0008-414X},
	MRCLASS = {42C15 (47B38)},
	MRNUMBER = {1350650},
	MRREVIEWER = {Eugenio Hern\'{a}ndez},
	DOI = {10.4153/CJM-1995-056-1},
	URL = {https://doi.org/10.4153/CJM-1995-056-1},
}

@article{T04,
	AUTHOR = {Tao, Terence},
	TITLE = {Fuglede's conjecture is false in 5 and higher dimensions},
	JOURNAL = {Math. Res. Lett.},
	FJOURNAL = {Mathematical Research Letters},
	VOLUME = {11},
	YEAR = {2004},
	NUMBER = {2-3},
	PAGES = {251--258},
	ISSN = {1073-2780},
	MRCLASS = {42B99 (43A45 46E30 52C22)},
	MRNUMBER = {2067470},
	MRREVIEWER = {B\'{e}la Uhrin},
	DOI = {10.4310/MRL.2004.v11.n2.a8},
	URL = {https://doi.org/10.4310/MRL.2004.v11.n2.a8},
}

\end{document}